\documentclass[12pt]{amsart}
\usepackage{hyperref}

\usepackage{amsmath,amssymb,amsthm,amscd, eucal}
\usepackage{amscd,enumerate} 
\usepackage{latexsym}

\newtheorem{thm}[equation]{Theorem}

\newtheorem{cor}[equation]{Corollary}
\newtheorem{lem}[equation]{Lemma}
\newtheorem{proposition}[equation]{Proposition}

\theoremstyle{definition}

\newtheorem{pgraph}[equation]{{}}
\numberwithin{equation}{section}

\newcommand{\heart}{{\rm heart}}

\newcommand{\cL}{{\mathcal{L}}}
\newcommand{\cR}{{\mathcal{R}}}

\newcommand{\cI}{{\mathcal{I}}}

\newcommand{\soc}{{\rm soc}}

\newcommand{\lann}{{\rm lann}}
\newcommand{\rann}{{\rm rann}}
\newcommand{\leftann}{{\rm LeftAnn}}

\begin{document}

\title[Maximal zero product subrings]
{Maximal zero product subrings and inner ideals of simple rings}

\author[A. Baranov]{Alexander Baranov}
\address{Department of Mathematics, University of Leicester, Leicester LE1 7RH, UK}
\email{ab155@le.ac.uk}
\thanks{Supported by University of Leicester}

\author[A. Fern\'andez L\'opez]{Antonio Fern\'andez L\'opez}
\address{Departamento de \'{a}lgebra, Geometr\'{\i}a y Topolog\'{\i}a,
Universidad de M\'{a}laga, 29071, M\'{a}laga, Spain}
 \email{emalfer@uma.es}
\thanks{Supported by the Spanish MEC and
Fondos FEDER, MTM2014-52470-P}





\begin{abstract}
Let $Q$ be a (not necessarily unital) simple ring or algebra. A nonempty
subset $S$ of $Q$ is said to have {\em zero product} if $S^2 =0$.
We classify all maximal zero product subsets of $Q$ by proving that
the map $\cR  \mapsto \cR  \cap \leftann(\cR )$ is a
bijection from the set of all proper nonzero annihilator right
ideals of $Q$ onto the set of all maximal zero product subsets of
$Q$.  We also describe the relationship between the maximal zero product
subsets of $Q$ and the maximal inner ideals of its
associated Lie algebra.
\end{abstract}

\maketitle

\section{Introduction}

Let $Q$ be a (not necessarily unital) associative ring or algebra. A nonempty
subset $S$ of $Q$ is said to have {\em zero product} if $S^2 =0$.
By Zorn's Lemma, any zero product subset is contained in a maximal
one, which is obviously a zero product subring. Note also that $0$
is the unique maximal zero product subset of a ring $Q$ if and
only if $Q$ has no nonzero nilpotent elements.

In this paper we describe the maximal zero product subsets of a
prime ring $Q$ with nonzero heart, in particular, of a simple ring,
by proving that the map $\cR  \mapsto \cR  \cap \leftann(\cR )$ is a
bijection from the set of all proper nonzero annihilator right
ideals of $Q$ onto the set of all maximal zero product subsets of
$Q$. In particular, if $Q$ is a simple unital Baer ring (e.g. a
simple Artinian ring), all maximal zero product subsets of $Q$ are of the
form $eQ(1-e)$, where $e$ is a nontrivial idempotent of $Q$.
Moreover, if $e_1$ and $e_2$ are idempotents of $Q$ then 
$e_1Q(1-e_1)= e_2Q(1-e_2)$ if and only if $e_1e_2=e_2$ and $e_2e_1=e_1$ (equivalently, $e_1Q=e_2Q$).

In the case when $Q$ is a simple ring coinciding with its socle,
we classify the maximal zero product subsets of $Q$ in terms of the associated geometry.

Finally, we describe the relationship between the maximal zero product
subsets of a simple ring and the inner ideal structure of its
associated Lie algebra.

For simplicity of exposition, all results in the paper are stated for rings, but it is easy to see that they also hold for algebras over a field.

\section{Preliminaries and notation}
Throughout the paper, $Q$ is a (not necessarily unital)
associative  ring (or algebra); $\cL$ denotes a left and $\cR$ a right
ideal of $Q$; ${\cI}_r(Q)$ and ${\cI}_l(Q)$ are the lattices of
all right and left ideals of $Q$, respectively. By an ideal we
mean a two-sided ideal.

\begin{pgraph}\label{anideal}
For a nonempty subset $S$ of $Q$ we denote by
$$
\lann(S)=\leftann(S): =\{ a \in Q : aS  =0\}
$$
the left annihilator of $S$. Note that $\lann(S)$ is a left ideal
of $Q$ (an ideal if $S$ is a left ideal).
A left ideal
$\cL$ is said to be an {\em annihilator left ideal} if $\cL=\lann(S)$ for
some nonempty subset $S$ of $Q$. Similarly, one defines the right
annihilator $\rann(S): =\{ a \in Q : Sa =0\}$, which is called an {\em
annihilator right ideal}. Note that $\lann(S)=\lann(T)$
where $T=S+SQ$ is the right ideal of $Q$ generated by $S$. Similarly,
$\rann(S)=\rann(S+QS)$. 
\end{pgraph}

\begin{pgraph}
A ring $Q$ is said to be {\em semiprime} if $I^2=0$ implies $I=0$
for any ideal $I$ of $Q$; equivalently, $aQa=0$ implies $a=0$ for
every $a \in Q$. If $Q$ is semiprime, then $\lann(I) =\rann(I)$
and $I \cap \rann(I)=0$ for any ideal $I$ of $Q$.
\end{pgraph}

\begin{pgraph} \label{prime}
A ring $Q$ is said to be {\em prime} if $IJ=0$ implies $I=0$ or
$J=0$ for $I, J$ ideals of $Q$. For a ring $Q$ the following
conditions are equivalent:

\begin{itemize}
\item[(i)] $Q$ is prime.

\item[(ii)] $\lann(I)=0$  for any nonzero  ideal $I$ of $Q$.

\item[(iii)] $aIb=0$ implies $a=0$ or $b=0$, for any nonzero ideal
$I$ of $Q$ and any $a, b$ in $Q$.
\end{itemize}
\end{pgraph}

\begin{pgraph}\label{hh} Let $Q$ be a ring. The {\em heart} of $Q$, denoted
by $\heart(Q)$, is defined as the intersection of all nonzero
ideals of $Q$. 
Clearly, if $Q$ is simple then $\heart(Q)=Q$. 
If $Q$ has nonzero heart, then $\heart(Q)$ is a
minimal ideal. Moreover, a prime ring has nonzero heart if and only
if it contains a minimal ideal.  If $Q$ is prime with
nonzero socle $\soc(Q)$ (the sum of all minimal left ideals), then by 
\cite[Theorem III.3.1]{J}, $\soc(Q)$ is a simple ideal of $Q$, contained in any nonzero ideal of $Q$, so 
$\heart(Q) =\soc(Q)$. 
\end{pgraph}

\section{Orthogonal pairs of one-sided ideals}

\begin{pgraph} \label{gc} \
We have a Galois connection between the lattice
${\cI}_r(Q)$ of all right ideals of $Q$ and the lattice
${\cI}_l(Q)$ of all left ideals of $Q$ given by
${\cR} \mapsto \lann(\cR )$ and $\cL  \mapsto \rann({\cL})$, that is,

\begin{itemize}
\item[(i)] $\cL_1 \subseteq \cL_2 \Rightarrow
\rann(\cL_2) \subseteq \rann(\cL_1)$ \ and \
$\cR_1 \subseteq \cR_2 \Rightarrow
\lann(\cR_2) \subseteq \lann(\cR_1)$,

\item[(ii)] $\cL  \subseteq \lann(\rann(\cL ))$ \
and \ $\cR  \subseteq \rann(\lann(\cR ))$,
\end{itemize}
for all $\cL , \cL_1, \cL_2 \in
{\cI}_l(Q)$ and $\cR , \cR_1, {\cR}_2 \in {\cI}_r(Q)$.

\medskip
Denote by $\overline\cL := \lann(\rann(\cL ))$ and
$\overline\cR := \rann(\lann(\cR ))$ the
corresponding {\em closures} of $\cL$ and $\cR$. 
We say that $\cL$ (resp. $\cR$) is {\em closed} if  
$\cL=\overline\cL$ (resp. $\cR=\overline\cR$).
It follows from (i) and (ii) that
$$
\rann(\cL) \subseteq \overline{\rann(\cL )}=\rann(\lann(\rann(\cL )))=\rann(\overline{\cL})\subseteq \rann(\cL )
$$
and similarly for $\lann(\cR )$. Therefore we have
\begin{itemize}
\item[(iii)] $\rann(\cL ) = \overline{\rann(\cL )}
 = \rann(\overline\cL )$,
\item[(iv)] $\lann(\cR ) = \overline{\lann(\cR )}
 = \lann(\overline\cR )$.
\end{itemize}
In particular, 
\begin{itemize}
\item[(v)] 
a right (resp. left) ideal is closed if
and only if it is an annihilator right (resp. left) ideal.
\end{itemize}
\end{pgraph}

\begin{pgraph}
By an {\em orthogonal pair} of $Q$ we mean  a pair $({\cR}, \ \cL
)$, where $\cR$ is a nonzero right and $\cL$ is a nonzero left
ideals of $Q$ such that ${\cL}\cR  =0$.
\end{pgraph}

\begin{lem} \label{maxorth1} For an orthogonal pair
$(\cR , \ \cL)$ the following conditions
are equivalent:

\begin{itemize}
\item[(i)] $\cR  = \overline\cR $ and ${\cL} = \lann(\cR )$,

 \item[(ii)] $\cR  = \rann({\mathcal
L})$ and $\cL  = \lann(\cR )$,

\item[(iii)] $\cL  = \overline\cL $ and ${\cR} = \rann(\cL )$.
\end{itemize}
\end{lem}

\begin{proof}
We will prove (i)$\Leftrightarrow$(ii). The proof of (ii)$\Leftrightarrow$(iii) is similar.
Suppose that  $\cL  = \lann(\cR )$. Then $\rann(\cL )=\rann(\lann(\cR )=\overline\cR$, 
so that $\cR  = \overline\cR \Leftrightarrow \cR  = \rann({\mathcal L})$.
\end{proof}

Given two orthogonal pairs
$(\cR_1 , \ \cL_1)$ and  $(\cR_2 , \ \cL_2)$, 
we say that $(\cR_1 , \ \cL_1)\subseteq (\cR_2 , \ \cL_2)$
if and only if $\cR_1\subseteq \cR_2$ and $\cL_1\subseteq \cL_2$.
This gives a partial order on the set of orthogonal pairs.

\begin{proposition} \label{maxorth2} Let  $(\cR , \ \cL )$ be an orthogonal pair
of $Q$. Then the following hold.

\begin{itemize}
\item[(i)] $(\overline\cR , \ \lann(\cR ))$ and
$(\rann(\cL ), \ \overline\cL )$ are maximal
orthogonal pairs.

\item[(ii)] $(\cR , \ \cL )$ is contained in the maximal 
orthogonal pairs $(\overline\cR , \ \lann(\cR ))$ and $(\rann(\cL ), \ \overline\cL )$.

\item[(iii)] $(\cR , \ \cL )$ is maximal if and
only if it satisfies the equivalent conditions of Lemma
\ref{maxorth1}.
\end{itemize}
\end{proposition}

\begin{proof} (i) Since
$\cR\subseteq\overline\cR$ and $\cL
\subseteq \lann(\cR )$, both $\overline\cR $ and
$\lann(\cR )$ are nonzero; and since $\lann(\cR )
=\lann(\overline\cR )$, we have that $(\overline{\cR}, \ \lann(\cR ))$ is an orthogonal pair. Suppose now
that $(\overline\cR , \ \lann(\cR ))$ is contained
in an orthogonal pair $({\cR'} , \ {\cL'})$. Then
$\overline\cR  \subseteq {\cR'}$ implies $\lann({\cR'})
\subseteq \lann(\overline\cR ) = \lann({\cR})$,
so
$$
{\cL'} \subseteq \lann({\cR'}) \subseteq
\lann(\cR ) \subseteq {\cL'},
$$
 which proves that
$\lann(\cR ) = {\cL'}$. Hence
$${\cR'}
\subseteq \rann({\cL'}) = \overline\cR  \subseteq
{\cR'},
$$
 which proves that $\overline\cR  =
{\cR'}$. Therefore  the orthogonal pair
$(\overline\cR , \ \lann(\cR ))$ is maximal. Similarly, one can prove
 that $(\rann(\cL ), \ \overline\cL )$ is a maximal orthogonal pair.

(ii) 
As noted in the proof of (i), $(\cR , \ \cL )$ is
contained in the maximal orthogonal pair
$(\overline\cR , \ \lann(\cR ))$. Similarly,
$(\cR , \ \cL )$ is also contained in the maximal
orthogonal pair $(\rann(\cL ), \ \overline\cL )$.

(iii) Suppose that $(\cR , \ \cL )$ is maximal.
Then $(\cR , \ \cL ) \subseteq (\overline
\cR , \ \lann(\cR ))$ implies $\cR = \overline \cR $ and $\cL  = \lann(\cR )$.
\end{proof}

\begin{proposition} \label{trivinn1} Let $B$ be an additive subgroup of
$Q$. Then the following conditions are equivalent:

\begin{itemize}
\item[(i)] $BQB \subseteq B$ \ and  \ $B^2=0$.

\item[(ii)] There exist $\cL  \in {\cI}_l(Q)$ and $\cR  \in
{\cI}_r(Q)$ such that ${\cR}\cL  \subseteq B \subseteq \cR  \cap
{\cL}$ and $\cL \cR  =0$.
\end{itemize}
\end{proposition}

\begin{proof}

(i) $\Rightarrow$ (ii): Taking $\cL  = B + QB$ and $\cR  = B +
BQ$, it is easily seen that (ii) holds.

(ii) $\Rightarrow$ (i): Clearly, $B^2 \subseteq {\cL}\cR  =0$ and
$BQB \subseteq \cR Q\cL \subseteq \cR \cL  \subseteq B$.
\end{proof}

\begin{pgraph}\label{regdef} Following \cite{BR},
we say that an additive subgroup $B$ of $Q$ is
a  {\em regular inner ideal} of $Q$ if it satisfies
the equivalent conditions
of the above proposition. 
In that case the orthogonal pair $(\cR,\cL)$ in (ii) is said to be 
{\em associated to} $B$. 
We note the following properties of regular inner ideals.

\begin{itemize}
\item[{\rm (i)}]  If $B$ is nonzero in Proposition \ref{trivinn1}, then both $\cR$ and $\cL$ are
nonzero and therefore $(\cR, \cL)$ is an orthogonal pair.

\item[{\rm (ii)}]  If $Q$ is a prime ring and $(\cR, \cL)$ is an
orthogonal pair, then any additive subgroup $B$ of $Q$ with
${\cR}\cL \subseteq B \subseteq \cR  \cap {\cL}$ is a {\em
nonzero} regular inner ideal, since $B=0$ would imply $\cR Q \cL
\subseteq \cR \cL \subseteq B=0$, which is a contradiction by
\ref{prime}(iii).

\item[{\rm (iii)}]  If $Q$ is a von Neumann regular ring, then any
orthogonal pair $(\cR, \cL)$ gives rise to a {\em unique} regular
inner ideal $B=\cR \cL = \cR \cap \cL$, since 
$$\cR \cap \cL  \subseteq 
(\cR \cap \cL)Q(\cR \cap \cL) \subseteq \cR \cL \subseteq \cR \cap
\cL .$$
\end{itemize}
\end{pgraph}

\section{Zero product subrings of prime rings with nonzero heart}

The following lemma shows that the orthogonal pair associated to a nonzero regular inner ideal of 
 a prime ring with nonzero heart (see \ref{hh}) is 
defined almost uniquely.

\begin{lem} \label{simplering1} 
Let $Q$ be a prime ring  
with nonzero heart $H=\heart(Q)$ and let 
$B$ be a nonzero regular inner ideal of
$Q$ with associated orthogonal pair $(\cR , \ {\cL})$. Then $BH
=\cR H$ and $HB = H\cL $. In particular, if $Q$ is simple and
unital, then $BQ =\cR$ and $QB = \cL $.
\end{lem}

\begin{proof} By \ref{prime}(iii), the ideal $\cL H\cR$ is
nonzero. Since $\cL H\cR  \subseteq H$ and $H$ is minimal, we have
that $\cL H\cR  =H$. Hence
$$
\cR H = \cR \cL H\cR \subseteq BH\cR  = BH \subseteq \cR H,
$$
which proves that $BH =\cR H$. Similarly, one proves that $HB =
H\cL $.
\end{proof}

\begin{lem} \label{simplering2} 
Let $Q$ be a prime ring  
with nonzero heart $H$ and 
let $(\cR_1, \ {\cL}_1)$ and $(\cR_2, \ \cL_2)$ be maximal
orthogonal pairs in $Q$. Then $\cR_1 \cap \cL_1
\subseteq \cR_2 \cap \cL_2$ implies $({\cR}_1, \ \cL_1)= (\cR_2, \ \cL_2)$.
\end{lem}

\begin{proof} Set  $\cL  := \cL_1 + \cL_2$ and
$B_j =\cR_j \cap \cL_j$, $j=1, 2$. By Lemma \ref{simplering1}
applied to $B_1$ and $B_2$, 
which are regular inner ideals by (\ref{regdef})(ii),
$$H\cL  = H\cL_1 + H\cL_2 = HB_1 + HB_2
\subseteq HB_2 = H\cL_2 \subseteq \cL_2.$$ We claim that $\cL
\cR_2 =0$. Otherwise, $H=H \cL \cR_2$ (as $\cL\cR_2$ is a
two-sided ideal and $H$ is minimal) and hence, by the formula
displayed above,
$$
H =H\cL \cR_2 \subseteq {\cL}_2\cR_2 =0,
$$
 which is a contradiction. Thus ${\cL}\cR_2 =0$ and hence
 $$
 \cL_1 \subseteq {\cL} \subseteq \lann(\cR_2) =\cL_2
 $$
by Proposition \ref{maxorth2}(iii).  Similarly,
$\cR_1 \subseteq \cR_2$. But then $({\cR}_1, \ \cL_1)= (\cR_2, \
\cL_2)$ by maximality of $({\cR}_1, \ \cL_1)$.
\end{proof}

Now we are ready to prove our main result, which describes 
maximal zero product subsets of prime rings with nonzero hearts (in particular, of simple rings). 

\begin{thm}\label{equiv}
Let $Q$ be a prime ring with nonzero heart containing nonzero nilpotent
elements and let $S$ be a subset of $Q$. 
Then the following are equivalent.
\begin{itemize}
\item[{\rm (i)}]  $S$ is a maximal zero product subset of $Q$.

\item[{\rm (ii)}]  $S$ is a maximal regular inner ideal of $Q$.

\item[{\rm (iii)}] $S=\cR  \cap \cL $, where $(\cR , \
\cL )$ is a maximal orthogonal pair, i.e. $\cR=\rann(\cL)$
and $\cL=\lann(\cR)$.
\end{itemize}
\end{thm}

\begin{proof}
(i)$\Rightarrow$(ii): Suppose that $S$ is a maximal zero product
subset. Note that $S$ is non-zero as $Q$ contains nonzero nilpotent elements. 
Since span of $S$ is a zero product subset,
$S$ is actually an additive subgroup of $Q$. Put $B=SQS+S$. Then
$S\subseteq B$ and $B^2=0$. Since $S$ is maximal, one has $S=B$.
Therefore
$$
BQB=SQS\subseteq B,
$$
so $B=S$ is a regular inner ideal of $Q$. Clearly, $B$ is maximal as $S$ is maximal.

(ii)$\Rightarrow$(iii): Suppose that $S$ is a maximal regular inner ideal of $Q$. 
Note that $S$ is non-zero (otherwise, it is strictly contained in a non-zero maximal 
zero product subset $S'$ of $Q$, which is a regular inner ideal of $Q$ as (i)$\Rightarrow$(ii)). 
By Proposition \ref{trivinn1}, there is an
orthogonal pair $(\cR , \ \cL )$ such that $S\subseteq
\cR\cap\cL$. By Proposition \ref{maxorth2}(ii), one can assume that
the pair $(\cR , \ \cL )$ is maximal. By Proposition
\ref{trivinn1}, $\cR \cap\cL$ is a regular inner ideal of $Q$, so
$S=\cR \cap\cL$ as $S$ is maximal.

(iii)$\Rightarrow$(i):
Let  $S=\cR  \cap \cL $ where $(\cR , \
\cL )$ is a maximal orthogonal pair.
Then $S^2\subseteq \cL\cR=0$, so $S$ is a zero product subset.
Let $S'$ be a maximal zero product subset of $Q$ containing $S$.
By the implication (i)$\Rightarrow$(iii), already established, 
 $S'=\cR'  \cap \cL' $ where $(\cR' , \
\cL' )$ is a maximal orthogonal pair.
By Lemma \ref{simplering2}, $(\cR , \ \cL )=(\cR' , \ \cL' )$,
so $S=S'$, as desired.
\end{proof}

\begin{cor}\label{bijection}
Let $Q$ be as in Theorem \ref{equiv}. Then the map $\cR  \mapsto \cR  \cap \lann(\cR )$ (resp. $\cL  \mapsto \cL  \cap \rann(\cL )$) 
is a bijection from
the set of all
 proper nonzero annihilator right (resp. left) ideals of $Q$ onto the set of all maximal zero product subsets of $Q$.
\end{cor}

\begin{proof} 
Let $\cR$ be a proper nonzero annihilator right ideal of $Q$. 
Then  $\cR= \rann(S)$ for some non-empty subset $S$ of $Q$. 
Note that $S\ne \{0\}$ as $\cR\ne Q$. 
Since the left ideal $\lann(\cR )$ contains $S$, it is non-zero. 
By (\ref{gc})(v), $\cR$ is closed, so by Lemma \ref{maxorth2}(ii), 
$(\cR,\lann(\cR ))$ is a maximal orthogonal pair of $Q$. Therefore,
by Theorem \ref{equiv}, $\cR\cap\lann(\cR )$ is a maximal zero product subset of $Q$.
We have shown that the map makes sense. Now Lemma \ref{simplering2} shows that the map is
injective and Theorem \ref{equiv} proves that it is surjective too. 
\end{proof}

\begin{pgraph} Recall that $Q$ is a \emph{Baer ring} if every left
annihilator of any subset of $Q$ is generated (as a left ideal) by
an idempotent element. If $Q$ is unital then it is known that this
definition is left-right symmetric. Note that every simple
Artinian ring is a unital Baer ring.
\end{pgraph}

\begin{cor} \label{corequiv} Let $Q$ be a simple unital Baer ring with nonzero
nilpotent elements.
Then $S \subset Q$ is a maximal zero product subset if and only if
$S=eQ(1-e)$ where $e\ne0,1$ is a nontrivial idempotent of $Q$.
Moreover, if $e_1$ and $e_2$ are idempotents of $Q$ then 
$e_1Q(1-e_1)= e_2Q(1-e_2)$ if and only if $e_1e_2=e_2$ and $e_2e_1=e_1$ (equivalently, $e_1Q=e_2Q$).
\end{cor}

\begin{proof}
By Theorem \ref{equiv} and Corollary \ref{bijection}, 
the  maximal zero product subsets of $Q$
are the intersections $\cR  \cap \lann(\cR )$ where $\cR$ runs
over all proper nonzero annihilator right ideals of $Q$. Since $Q$
is Baer, $\cR=eQ$ for some idempotent $e$. Then
$\lann(\cR) = \lann(eQ) =Q(1-e)$.
Indeed, one has $a \in \lann(eQ)$ if and only if
$ae =0$, or equivalently,
$a=a(1-e) \in Q(1-e)$.
It is
also clear that
$$\cR  \cap \lann(\cR ) = eQ \cap Q(1-e) =eQ(1-e).$$
Finally, by Corollary \ref{bijection}, $e_1Q(1-e_1)= e_2Q(1-e_2)$
if and only if 
$e_1Q =e_2Q$. It is easy to see that the latter condition is equivalent
to $e_1e_2=e_2$ and $e_2e_1=e_1$. 
\end{proof}

\begin{pgraph} \label{onesided} By \cite[IV.8]{J}, a ring $Q$
 is simple with minimal one-sided ideals if and only $Q\cong Y \otimes_\Delta X$, where
$(X, Y, \langle \cdot , \cdot \rangle)$ is a pair of dual vectors
spaces over a division ring $\Delta$, and where the product is
given by
$$(y_1 \otimes x_1)(y_2 \otimes x_2) = y_1 \otimes \langle x_1, y_2\rangle
x_2$$ for all $x_1, x_2 \in X$, $y_1, y_2 \in Y$. According to
this representation of $Q$, we have (see \cite[IV.16.Theorem
1]{J}:
\begin{itemize}
\item[{\rm (i)}] The map $W \mapsto W \otimes X$ is a lattice
isomorphism of the lattice ${\mathcal S}(Y)$ of all subspaces of
$Y$ onto the lattice ${\cI}_r(Q)$ of all right ideals of $Q$.

\item[{\rm (ii)}] The map $V \mapsto Y \otimes V$ is a lattice
isomorphism of the lattice ${\mathcal S}(X)$ of all subspaces of
$X$ onto the lattice ${\cI}_l(Q)$ of all left ideals of $Q$.
\end{itemize}
\end{pgraph}

\begin{pgraph} \label{annih}
It is easy to check that if $\cR =W \otimes X$ is a right ideal of
$Q$, then $\lann(\cR) = Y \otimes W^\perp$, where $W^\perp = \{ x
\in X : \langle x, W \rangle =0\}$. Similarly, for any left ideal
$\cL =Y \otimes V$ of $Q$, $\rann(Y \otimes V) = V^\perp \otimes
X$. Thus $\cR$ is an annihilator right ideal if and only $\cR = W
\otimes Y$, where $W$ is a {\em closed} subspace of $Y$, i.e.,
$W^{\perp \perp} =W$. 

We have a Galois connection between the
lattice ${\mathcal S}(X)$ of all subspaces of $X$ and the lattice
${\mathcal S}(Y)$ of all subspaces of $Y$ given by $V \to V^\perp$
and $W  \to W^\perp$. 

It is easy to see that all finite dimensional subspaces of $X$ (resp. $Y$) are closed. 
Indeed, let $V$ be a finite-dimensional subspace of $X$ with basis  $\{ x_1, \ \dots, \ x_n\}$. 
Fix a dual set $\{ y_1, \ \dots, \ y_n\}$ of vectors of $Y$ such that 
$\langle x_i, y_j\rangle = \delta_{ij}$, $1\le i,j\le n$, and denote 
$W=\text{span}\{ y_1, \ \dots, \ y_n\}$.  
Then we have $X= V \oplus W^\perp$ and $Y=W \oplus V^\perp$. 
Clearly, $V \subseteq V^{\perp\perp}$. 
Thus, $V \ne V^{\perp\perp}$ if and only if there exists a non-zero $u\in V^{\perp\perp}\cap W^\perp$.
But we have 
$\langle u, Y\rangle \subseteq \langle u, W\rangle + \langle u, V^\perp \rangle =0$, 
which implies, by nondegeneracy, that $u=0$. 
\end{pgraph}

\begin{cor} \label{reginner}
Let $Q=Y\otimes_\Delta X$ be a simple ring with minimal one-sided
ideals. Then the map $W \mapsto W\otimes W^\perp$ is a bijection
from the set of nonzero proper closed subspaces of $Y$ onto the
set of maximal zero product subsets of $Q$.
\end{cor}

\begin{proof} By Corollary \ref{corequiv}, any maximal zero product subset
$S$ of $Q$ is of the form $S={\cR} \cap \lann({\cR})$ for a unique
proper nonzero annihilator right ideal $\cR$ of $Q$. Now it
follows from (\ref{onesided}) and (\ref{annih}) that ${\cR} = W
\otimes X$ for a unique nonzero proper closed subspace $W$ of $Y$
and  $\lann(W \otimes X) = Y \otimes W^\perp$.  Hence
$$
S={\cR}
\cap \lann({\cR}) = (W \otimes X) \cap (Y \otimes W^\perp) = W
\otimes W^\perp
$$
as required.
\end{proof}

 We finish with an application to the
Lie inner ideal structure of simple rings.

\begin{pgraph} Recall that every associative ring $Q$ becomes a Lie
ring $Q^{(-)}$ under the product $[x,y]=xy-yx$.
An additive subgroup $B$ of a Lie ring $L$ is called an {\em inner ideal} if $[[B, L], B] \subseteq B$. An inner ideal $B$
is said to be {\em abelian} if $[B, B] =0$. An inner ideal $B$ of $Q^{(-)}$
is said to be {\em Jordan-Lie} if $B^2=0$, see \cite{F,BS}. 
Inner ideals were first systematically studied by Benkart \cite{B}, see also
\cite{BS,BF,BFG,Fbook} for some recent development. 
\end{pgraph}

\begin{cor} Let $Q$ be a simple associative ring. 
For  an additive subgroup $B$ of
$Q$ the following conditions are equivalent.

\begin{itemize}
\item[{\rm (i)}] $B$ is a maximal zero product subset of $Q$.

\item[{\rm (ii)}] $B$ is a maximal regular inner ideal of $Q$.

\item[{\rm (iii)}] $B$ is a maximal Jordan-Lie inner ideal of
$Q^{(-)}$.
\end{itemize}

Moreover, if in addition, $Q$ is not unital (i.e. doesn't contain an identity element) 
and ${\rm char}(Q)\ne 2,3$ then 
the conditions (i)-(iii) are equivalent to 

\begin{itemize}
\item[{\rm (iv)}] $B$ is a maximal abelian inner ideal of
$Q^{(-)}$.
\end{itemize}

\end{cor}

\begin{proof} (i) $\Leftrightarrow$ (ii). This is proved in Theorem
\ref{equiv}.

(ii) $\Leftrightarrow$ (iii). Suppose $B$ is a  maximal regular inner ideal
of $Q$. Then by definition, $B^2=0$, so $B$ is Jordan-Lie. 
Note that any Jordan-Lie inner ideal of $Q$ is contained in a maximal zero product subset of $Q$, 
which is a maximal regular inner ideal by Theorem  \ref{equiv}. 
Thus, $B$ is maximal as Jordan-Lie.  Conversely, if $B$ is a  maximal Jordan-Lie inner ideal
of $Q$, then it must be maximal regular by above. 

Suppose now that  $Q$ is not unital. 

(iii) $\Leftrightarrow$ (iv). 
If $B$ is a  Jordan-Lie inner ideal
of $Q$, then $[B, B] \subseteq B^2 =0$, so $B$ is abelian. 
On the other hand, by \cite[Theorem 5.4]{F} (applied to the
case of a non-unital simple ring),  every abelian inner
ideal of $Q$ is Jordan-Lie. 
\end{proof}

\end{document}